\documentclass[11pt]{amsart}

\usepackage{amsmath,amssymb,amsthm}
\usepackage{graphicx,epsfig}
\usepackage{enumerate}
\usepackage{tikz}
\usepackage{csquotes}
\usepackage{multicol}
\usepackage{hyperref}
\usetikzlibrary{calc}
\usepackage{url}
\usepackage[]{algorithm2e}
\DeclareMathOperator{\reg}{reg}

\DeclareMathOperator{\Tor}{Tor}
\DeclareMathOperator{\Z}{\mathbb{Z}}
\DeclareMathOperator{\N}{\mathbb{N}}

\DeclareMathOperator{\coker}{coker}

\usepackage{comment}

\newtheorem{thm}{Theorem}[section]

\theoremstyle{definition}
\newtheorem{dfn}[thm]{Definition}

\newtheorem{rem}[thm]{Remark}

\newtheorem{example}[thm]{Example}

\newtheorem{question}[thm]{Question}

\usepackage[retainorgcmds]{IEEEtrantools}
\newcommand{\IE}[1]{\begin{IEEEeqnarray*}{#1}}
\newcommand{\IEE}{\end{IEEEeqnarray*}}

\begin{document}

\title[LinearTruncations.m2]{Linear Truncations Package for Macaulay2}
\author{Lauren Cranton Heller}
\address{Department of Mathematics, University of California, 970 Evans Hall, Berkeley, CA 94720, USA}
\email{lch@math.berkeley.edu}
\author{Navid Nemati}
\address{Universit\'{e} C\^{o}te d'Azur, Inria, 2004 route des Lucioles, 06902 Sophia Antipolis, France}
\email{navid.nemati@inria.fr}

\begin{abstract}
We introduce the Macaulay2 package $\mathtt{LinearTruncations}$ for finding and studying the truncations of a multigraded module over a standard multigraded ring that have linear resolutions.
\end{abstract}

\maketitle
\section{Introduction and preliminaries\texorpdfstring{\\
$\mathtt{multigradedPolynomialRing}$, $\mathtt{isLinearComplex}$, $\mathtt{supportOfTor}$}{}}\label{sec:notation}

\noindent Castelnuovo--Mumford regularity is a fundamental invariant in commutative algebra and algebraic geometry. Roughly speaking, it measures the complexity of a module or sheaf. Let $S$ be a polynomial ring with the standard grading and $M$ a finitely generated $S$-module. In this case Castelnuovo--Mumford regularity is typically defined in terms of either the graded Betti numbers of $M$ or the vanishing of local cohomology modules $H^i_\mathfrak{m}(M)$, where $\mathfrak{m}$ is the maximal homogeneous ideal of $S$. Eisenbud and Goto show in \cite{EisenbudGoto} that the Castelnuovo--Mumford regularity of $M$ is the minimum degree where the truncation of $M$ has a linear resolution.

An extension of Castelnuovo--Mumford regularity to multigraded modules was introduced by Hoffman and Wang \cite{HoffmanWang} in a special case, and later by Maclagen and Smith \cite{MaclagenSmith} and Botbol and Chardin \cite{BotbolChardin} in a more general setting.  The multigraded regularity of a module is a region in $\Z^r$ rather than an integer.  It is invariant under positive translations and thus can be described by its minimal elements.  An affirmative answer to the following open question would reduce this to a finite computation.

\begin{question}\label{ques:bound}
	Can the minimal elements of the regularity of $M$ be bounded in terms of $S$ and the Betti numbers of $M$?
\end{question}

In analogy to the singly graded case, one may ask about the relation between multigraded Castelnuovo--Mumford regularity and the multidegrees where the truncation of a module has a linear resolution, which we call the \textit{linear truncation region}.  (See Definitions \ref{def:linear} and \ref{def:region}.)  In the multigraded setting these regions can differ, but a bound on the linear truncations would answer Question~\ref{ques:bound} by \cite[Proposition~4.11]{EisenbudErmanSchreyer}.  Our goal is to compute the minimal elements of the linear truncation region within a specified finite region of $\Z^r$.

We introduce the \textit{LinearTruncations} package for \cite{M2}, which provides tools for studying the resolutions of truncations of modules over rings with standard multigradings.  Given a module and a bounded range of multidegrees, our package can identify all linear truncations in the range. The algorithm uses a search function that is also applicable to other properties of modules described by sets of degrees.  The examples here were computed using version 1.18 of Macaulay2 and version 1.0 of \textit{LinearTruncations}.  

In section \ref{Section Theorems} we describe the main algorithms of this package, $\mathtt{findRegion}$ and $\mathtt{linearTruncations}$. 
In Section \ref{sec: relation}, we discuss the relation between the linear truncation region and the multigraded regularity and we introduce $\mathtt{regularityBound}$ and $\mathtt{linearTruncationsBound}$ as faster methods for calculating subsets of the multigraded regularity and linear truncation regions, respectively.

To set our notation,  let $k$ be a field and 
\[S=k\left[x_{i,j} \mid 1\leq i\leq r, \, 0\leq j \leq n_i\right]\]
a $\Z^r$-graded polynomial ring with $\deg x_{ij}=\mathbf{e}_i$, the $i$-th standard basis vector in $\Z^r$, for all $j$ (so that $S$ is the coordinate ring of a product $\mathbb{P}^{n_1}\times\cdots\times\mathbb{P}^{n_r}$ of projective spaces).  The function $\mathtt{multigradedPolynomialRing}$ produces such rings:
\begin{verbatim}
    i1 : needsPackage "LinearTruncations"
    o1 = LinearTruncations
    o1 : Package
    i2 : S = multigradedPolynomialRing {1,2}
    o2 = S
    o2 : PolynomialRing
    i3 : degrees S
    o3 = {{1, 0}, {1, 0}, {0, 1}, {0, 1}, {0, 1}}
    o3 : List
\end{verbatim}
Let $M$ be a finitely generated $\Z^r$-graded $S$-module.  For a multidegree $\mathbf{d}=(d_1,\dots,d_r)\in \Z^r$, write $\bar{\mathbf{d}}$ for the total degree $d_1+\cdots +d_r$ of $\mathbf{d}$ and  $M_{\geq \mathbf{d}}$ for the truncation $\oplus _{\mathbf{d}'\geq\mathbf{d}} M_{\mathbf{d}'}$ of $M$ at $\mathbf{d}$, where $\mathbf{d}'\geq\mathbf{d}$ if this inequality is true for each coordinate.
\begin{dfn}\label{def:linear}
	A homogeneous chain complex
	$$0\leftarrow G_0\leftarrow G_1\leftarrow\cdots\leftarrow G_k\leftarrow 0$$
	of free $S$-modules is \emph{linear} if $G_0\simeq\bigoplus S(-\mathbf{d})$ for some $\mathbf{d}\in\Z^r$ and for each free summand $S(-\mathbf{d}')$ of $G_i$ we have $\bar{\mathbf{d}}'=\bar{\mathbf{d}}+i$.
\end{dfn}
The function $\mathtt{isLinearComplex}$ checks this condition.  To print the degrees appearing in the complex use $\mathtt{supportOfTor}$.
\begin{verbatim}
    i4 : B = irrelevantIdeal S
    o4 = ideal (x   x   , x   x   , x   x   , x   x   , x   x   , x   x   )
                 0,1 1,2   0,0 1,2   0,1 1,1   0,0 1,1   0,1 1,0   0,0 1,0
    o4 : Ideal of S
    i5 : F = res comodule B
           1      6      9      5      1
    o5 = S  <-- S  <-- S  <-- S  <-- S  <-- 0
                                          
         0      1      2      3      4      5
    o5 : ChainComplex
    i6 : netList supportOfTor F
\end{verbatim}
\pagebreak
\begin{verbatim}
         +------+------+
    o6 = |{0, 0}|      |
         +------+------+
         |{1, 1}|      |
         +------+------+
         |{2, 1}|{1, 2}|
         +------+------+
         |{2, 2}|{1, 3}|
         +------+------+
         |{2, 3}|      |
         +------+------+
    i7 : isLinearComplex F
    o7 = false
\end{verbatim}
\begin{dfn}\label{def:region}
	The \emph{linear truncation region} of $M$ is
	\[\{ \mathbf{d} \mid M_{\geq\mathbf{d}} \,\text{has a linear resolution with generators in degree $\mathbf{d}$}\}\subset \Z^r.\]
\end{dfn}
\begin{rem}\label{rem: minimal}
	Our definitions imply that the nonzero entries in the differential matrices of a linear resolution will have total degree 1.  We also require that the generators have degree $\mathbf{d}$ so that a linear resolution for $M_{\geq\mathbf{d}}$ implies the existence of a linear resolution for $M_{\geq\mathbf{d}'}$ whenever $\mathbf{d}'\geq\mathbf{d}$.  We can thus describe the linear truncation region by giving its minimal elements.
\end{rem}

\section{Finding linear truncations\texorpdfstring{\\ $\mathtt{linearTruncations}$, $\mathtt{findMins}$, $\mathtt{findRegion}$}{}}\label{Section Theorems}

\noindent Eisenbud, Erman, and Schreyer proved in \cite{EisenbudErmanSchreyer} that the linear truncation region of $M$ is non-empty.  In particular it contains the output of the function $\mathtt{coarseMultigradedRegularity}$ from their package \emph{TateOnProducts} \cite{TateOnProducts}.  However, in general this degree is neither a minimal element itself nor greater than all the minimal elements.  (See Example~\ref{ex: outside}.)

The function $\mathtt{linearTruncations}$ searches for multidegrees where the truncation of $M$ has a linear resolution by calling the function $\mathtt{findRegion}$, which implements Algorithm~\ref{alg:findRegion}.  Since we do not know of a bound on the total degree of the minimal elements in the linear truncation region given the Betti numbers of $M$, $\mathtt{linearTruncations}$ is not guaranteed to produce all semigroup generators.  By default it searches above the componentwise minimum of the degrees of the generators of $M$ and below the degree with all coordinates equal to $r+1$, where $r$ is the output of $\mathtt{regularity}$.  Otherwise the range is taken as a separate input.
\begin{example}\label{ex: outside}
Let $S=k[x_{0,0},x_{0,1}, x_{0,2}, x_{1,0},x_{1,1},x_{1,2}, x_{1,3}]$ be the Cox ring of $\mathbb{P}^2\times \mathbb{P}^3$. For each $d\geq 2$, let $\phi_d\colon S(-d,-d)^6\rightarrow S(0,-d)^2\oplus S(-d,0)^4$ be given by 
$$
\begin{pmatrix}
x_{0,0}^d&x_{0,1}^d&x_{0,2}^d&0&0&0\\
0&0&0&x_{0,1}^d&x_{0,0}^d&x_{0,2}^d\\
x_{1,0}^d&0&0&x_{1,0}^d&0&0\\
0&x_{1,1}^d&0&0&x_{1,1}^d&0\\
0&0&x_{1,2}^d&0&0&x_{1,2}^d\\
0&0&0&x_{1,3}^d&0&0
\end{pmatrix},
$$
and define $M^{(d)} := \coker\phi_d$. The $\mathtt{coarseMultigradedRegularity}$ of $M^{(3)}$ is $\{3,3\}$, the $\mathtt{regularity}$ of $M^{(3)}$ is 5, and $\{3,3\}$ and $\{8,2\}$ are minimal elements of the linear truncation region.  Since $\{8,2\}$ is not below $\{5+1,5+1\}$ it will not be returned by the $\mathtt{linearTruncations}$ function with the default options.

Based on the computations from $M^{(d)}$ for $2\leq d\leq 10$ we expect that for $d\geq 2$ the module $M^{(d)}$ will have $\mathtt{coarseMultigradedRegularity}$ equal to $\{d,d\}$, with $\{d,d\}$ and $\{3d-1,d-1\}$ both minimal elements of the linear truncation region.
\end{example}

\begin{algorithm}\label{alg:findRegion}
	\SetKwInOut{Input}{Input}\SetKwInOut{Output}{Output}
	\Input{a module $M$, a Boolean function $f$, and a range $(\mathbf{a},\mathbf{b})$}
	\Output{minimal elements between $\mathbf{a}$ and $\mathbf{b}$ where $M$ satisfies $f$}
	$A:= \emptyset$\; $K := \{\mathbf{a}\}$\;
	\While{$K\neq\emptyset$}{
		$\mathbf{d} :=$ first element of $K$\;
		$K=K\setminus\{\mathbf{d}\}$\;
		\If{$\mathbf{d}\notin A+\N^r$}{
			\eIf{$M$ satisfies $f$ at $\mathbf{d}$}{
				$A=A\cup\{\mathbf{d}\}$\;
			}{
				\For{$1\leq i\leq r$}{
					\If{$\mathbf{d}+\mathbf{e}_i\leq\mathbf{b}$}{
						$K=K\cup\{\mathbf{d}+\mathbf{e}_i\}$\;
					}
				}
			}
		}
	}
	\Return {$A$}
	\caption{$\mathtt{findRegion}$}
\end{algorithm}

At each step of Algorithm~\ref{alg:findRegion} the set $A$ contains degrees satisfying $f$ and the set $K$ contains the minimal degrees remaining to be checked.  There are options to initialize $A$ and $K$ differently---degrees in $A$ will be assumed to satisfy $f$, and degrees below those in $K$ will be excluded from the search (and thus assumed not to satisfy $f$).  Supplying such prior knowledge can decrease the length of the computation by limiting the number of times the algorithm calls $f$.

The pseudocode in Algorithm~\ref{alg:findRegion} masks the fact that $A$ and $K$ are stored as monomial ideals in a temporary standard graded polynomial ring.  Similarly, the function $\mathtt{findMins}$ will convert a list of multidegrees to a monomial ideal in order to calculate its minimal elements via a Gr\"{o}bner basis.

\section{Relation to regularity\texorpdfstring{\\ $\mathtt{partialRegularities}$, $\mathtt{linearTruncationsBound}$, $\mathtt{regularityBound}$}{}}\label{sec: relation}

\noindent As discussed above, the minimal element of the linear truncation region of a singly graded module agrees with its Castelnuovo--Mumford regularity, which can be determined from its Betti numbers.  In the multigraded case these concepts are still linked, but their relationship is more complicated.  For instance, the following inclusion is strict:
\begin{thm}\label{thm:regularity}
	If $H^0_B(M)=0$, then the linear truncation region of $M$ is a subset of the multigraded regularity region $\reg M$ of M, as defined in \cite{MaclagenSmith}.
\end{thm}
\begin{proof}
	See \cite[Theorem~2.9]{BerkeschErmanSmith} or \cite[Proposition~4.11]{EisenbudErmanSchreyer}. 
\end{proof}
Unfortunately the multigraded Betti numbers of $M$ do not determine either its regularity or its linear truncations.  However, the functions $\mathtt{regularityBound}$ and $\mathtt{linearTruncationsBound}$ compute subsets of these regions using only the twists appearing in the minimal free  resolution of $M$.  In many examples they produce the same outputs as $\mathtt{multigradedRegularity}$ (from the package \emph{VirtualResolutions} \cite{VirtualResolutions}) and $\mathtt{linearTruncations}$, respectively, without computing sheaf cohomology or truncating the module.

The algorithms for $\mathtt{linearTruncationsBound}$ and $\mathtt{regularityBound}$ are based on the following theorem (from \cite{LJ}):
\begin{thm}\label{thm: LJ}
	If $H^0_B(M)=0$ and $H^1_B(M)=0$ then
	\[\bigcap_{\Tor_i(M,k)_\mathbf{b}\neq 0}\bigcup_{\sum\lambda_j=i}\left[\mathbf{b}-\lambda_1\mathbf{e}_i-\cdots-\lambda_r\mathbf{e}_r+\mathbb{N}^r\right]\]
	is a subset of the $\mathtt{linearTruncations}$ of $M$, and 
	\[\bigcap_{\Tor_i(M,k)_\mathbf{b}\neq 0}\bigcup_{\sum\lambda_j=i-1}\left[\mathbf{b}-\mathbf{1}-\lambda_1\mathbf{e}_i-\cdots-\lambda_r\mathbf{e}_r+\mathbb{N}^r\right]\]
	is a subset of the $\mathtt{multigradedRegularity}$ of $M$.
\end{thm}
The function $\mathtt{partialRegularities}$ calculates the  Castelnuovo--Mumford regularity in each component of a multigrading. 
\begin{rem}
In the bigraded case, Theorem~\ref{thm: LJ} implies that  $\mathbf{d}$ is in $\mathtt{linearTruncations\ M}$ if $\mathbf{d}\geq \mathtt{partialRegularities\ M}$  and $\bar{\mathbf{d}}\geq \mathtt{regularity\ M }$.
\end{rem}
For some modules $\mathtt{linearTruncationsBound}$ gives a proper subset of the linear truncations:
\begin{verbatim}
    i8 : S = multigradedPolynomialRing 2;
    i9 : M = coker(map(S^{{-1,0},{0,-1},{0,-1}},S^{{-1,-1},{-1,-1}},
    {{x_(1,0),x_(1,1)},{-x_(0,0),0},{0,-x_(0,1)}}));
    i10 : multigraded betti res M
                0   1
    o10 = 1: a+2b   .
          2:    . 2ab
    i11 : linearTruncations M
    o11 = {{0, 2}, {1, 1}}
    i12 : linearTruncationsBound M
    o12 = {{1, 1}}
\end{verbatim}

\section{Acknowledgments}
\noindent The authors would like to thank David Eisenbud for his comments.  The second author would like to thank the organizers of the Macaulay2 workshop in Leipzig (June 2018), during which the algorithms were first implemented.

\bibliographystyle{alpha}
\bibliography{bib.bib}

\end{document}